\documentclass[12pt,oneside]{amsart}

\pdfoutput=1

\usepackage[section]{placeins}
\usepackage[margin=1in]{geometry}
\usepackage{mathabx}
\usepackage{xcolor}
\usepackage{graphicx}
\usepackage{amssymb}
\usepackage{tikz}
\usepackage{latexsym}
\usepackage{epsfig}
\usepackage{enumitem}
\usepackage{multicol}
\usepackage{wasysym}
\usepackage{amsmath}
\usepackage{amsthm}
\usepackage{amscd}
\usepackage{ulem}
\usepackage{soul}
\usepackage{multirow}
\usepackage{rotating}
\usepackage{pdflscape}
\usepackage{float}
\usepackage{dsfont}
\usepackage{comment}

\usepackage[colorlinks,citecolor=red,pagebackref,hypertexnames=false]{hyperref}
\hypersetup{
  colorlinks = true,
  citecolor=red,
  linkcolor  = red
}

\setlist{itemsep=.06125in}

\restylefloat{table}
\numberwithin{equation}{section}

\theoremstyle{plain}
\newtheorem{theorem}{Theorem}[section]
\newtheorem{lemma}[theorem]{Lemma}
\newtheorem{corollary}[theorem]{Corollary}
\newtheorem{proposition}[theorem]{Proposition}

\theoremstyle{definition}
\newtheorem{definition}[theorem]{Definition}

\theoremstyle{remark}
\newtheorem{remark}[theorem]{Remark}

\newcommand{\eps}{\varepsilon}

\date{\today}

\author{A. Iosevich, J. Iosevich, E. Palsson, and A. Yavicoli}
\address{Department of Mathematics, University of Rochester, Rochester, NY, USA}
\email{iosevich@gmail.com}
\address{Department of Applied Mathematics, Rochester Institute of Technology, Rochester, NY, USA}
\email{joshuaiosevich@gmail.com}
\address{Department of Mathematics, Virginia Tech, Blacksburg, VA, USA}
\email{palsson@vt.edu}
\address{Department of Mathematics, University of British Columbia, Vancouver, BC, Canada}
\email{yavicoli@math.ubc.ca}

\thanks{A. I. was supported in part by the National Science Foundation under NSF DMS - 2154232.}
\thanks{A. Y. was supported in part by the Natural Sciences and Engineering Research Council of Canada, NSERC (GR030571 and GR030540).}

\title{PDE propagation, sampling, and the Fourier ratio}

\begin{document}

\begin{abstract}
We study recovery from incomplete random spatial samples for discretized fields arising as fixed-time snapshots of partial differential equations. The organizing parameter is the Fourier ratio
$$
FR(g)=\frac{\|\widehat g\|_1}{\|\widehat g\|_2},
$$
which quantifies effective spectral dimension and governs stable $\ell^1$ recovery in bounded orthonormal sampling models.

Our main observation is that fixed-time PDE propagation can strictly improve Fourier ratio bounds relative to the discretized initial data. In dimension three, the wave snapshot operator introduces additional high-frequency decay, leading after discretization to Fourier ratio bounds that are uniformly controlled in the grid size (up to discretization errors), whereas the corresponding bounds for the initial discretization are typically polynomial in $N$. For the heat equation in any dimension, Gaussian frequency damping yields Fourier ratio bounds that are essentially independent of grid resolution for fixed positive time.

Combining these deterministic Fourier ratio improvements with standard $\ell^1$ recovery guarantees yields explicit sampling-rate bounds for stable reconstruction from missing spatial samples. Numerical experiments confirm that PDE propagation acts as a spectral preconditioner that lowers effective sampling complexity in practice.
\end{abstract}



\maketitle

\section{Introduction}

The purpose of this paper is to extend the sampling framework of \cite{IPY2026} to basic PDE evolutions and to quantify how propagation changes sampling requirements for recovery from missing spatial data. The central observation is that fixed-time PDE propagators reduce the Fourier ratio of discretized fields. Because the number of random point samples required for stable $\ell^1$ recovery scales quadratically with the Fourier ratio (up to logarithmic factors), this reduction translates directly into lower sampling budgets. In dimension three, the fixed-time wave snapshot replaces a polynomial-in-$N$ sampling requirement for the discretized initial data by one that is uniformly controlled in $N$ for fixed positive time (up to discretization error terms). For the heat equation, the sampling requirement becomes essentially independent of the grid size for each fixed $t>0$.

From an imaging perspective, the results of this paper quantify how physical propagation mechanisms can reduce the sampling burden in reconstruction problems with missing data. In many sensing modalities the measured field is not the original signal but rather a propagated state governed by an underlying partial differential equation. Diffusion processes arise naturally in heat transfer, optical blur, and other smoothing phenomena, while wave propagation appears in acoustics, seismic imaging, and electromagnetic sensing. In such settings the observed data may be interpreted as a discretized snapshot of a PDE evolution. The analysis in this paper shows that this propagation can act as a spectral preconditioner: it suppresses high-frequency Fourier modes and thereby lowers the Fourier ratio of the discretized field. Because the sampling guarantees used in compressed sensing depend directly on this quantity, the effect is a reduction in the number of spatial samples required for stable recovery.

As in \cite{IPY2026}, the underlying sampling framework is the classical compressed sensing paradigm based on stable recovery from incomplete measurements via $\ell^1$ minimization and related convex programs, as initiated in \cite{CRT06} and developed further in, for example, \cite{FR13,Rau10,RV08}. A recurring theme is that we do not impose any sparsity hypothesis on the signal in physical space. Instead, compressibility is controlled by the Fourier ratio
$$
FR(g)=\frac{\|\widehat g\|_1}{\|\widehat g\|_2},
$$
which serves as a quantitative proxy for effective spectral dimension. This viewpoint fits naturally with uncertainty principle and structural perspectives in signal recovery, such as \cite{DS89,IM24}.

A key quantitative point in this paper is that we compare the Fourier ratio of a discretized PDE snapshot $g_t$ directly to the Fourier ratio of the discretized initial data $g$. The discretized initial data $g$ inherits only the Fourier decay coming from regularity of $f$. In contrast, the snapshot $g_t$ inherits both the regularity of $f$ and the additional frequency decay introduced by the propagator. In the wave case in three dimensions this replaces an a priori polynomial Fourier ratio bound for $g$ by a bound that is uniformly controlled in the grid size for fixed $t>0$ (up to discretization errors). In the heat case in any dimension, the Gaussian damping forces a Fourier ratio bound for $g_t$ that is essentially independent of the grid size for fixed positive time, while the corresponding Fourier ratio bound for $g$ is polynomial in $N$ in dimensions $d\ge 3$.

\subsection{Wave equation motivation}

Let $f$ be a real-valued function that is 1-periodic on $[0,1]^3$. Consider the solution $u$ of the wave equation
$$
u_{tt}=\Delta u,
\qquad
u(x,0)=0,
\qquad
u_t(x,0)=f(x).
$$
Write the Fourier series of $f$ as
$$
f(x)=\sum_{k\in\mathbb Z^d} a_k e^{2\pi i k\cdot x}.
$$
Then the periodic wave evolution satisfies
$$
u(x,t)= ta_0 + \sum_{k\in\mathbb Z^3\setminus\lbrace 0 \rbrace} \dfrac{\sin(2\pi t|k|)}{2\pi |k|}\, a_k e^{2\pi i k\cdot x}.
$$
Thus for fixed $t>0$ the mapping $f\mapsto u(\cdot,t)$ is a Fourier multiplier of order $-1$.

The guiding heuristic is that, after discretization on ${\mathbb Z}_N^3$, the additional factor $|k|^{-1}$ improves the high-frequency tail of the discrete Fourier coefficients. Since the Fourier ratio is an $\ell^1/\ell^2$ quantity, improved tail bounds for the discrete Fourier coefficients typically reduce $FR$ significantly.

\subsection{Heat equation motivation}

Let $f$ be a real-valued function that is 1-periodic on $[0,1]^d$. Consider the solution $v$ of the heat equation
$$
v_t=\Delta v,
\qquad
v(x,0)=f(x).
$$

Write the Fourier series of $f$ as
$$
f(x)=\sum_{k\in\mathbb Z^d} a_k e^{2\pi i k\cdot x}.
$$
Then the periodic heat evolution $v=e^{t\Delta}f$ satisfies
$$
v(x,t)=\sum_{k\in\mathbb Z^d} e^{-4\pi^2 t|k|^2}a_k e^{2\pi i k\cdot x}.
$$
Thus for fixed $t>0$ the mapping $f\mapsto v(\cdot,t)$ is a Fourier multiplier with symbol $e^{-4\pi^2 t|k|^2}$ on Fourier series coefficients, hence smoothing of infinite order.

After discretization on ${\mathbb Z}_N^d$, this Gaussian frequency damping forces the discrete Fourier coefficients to have an absolutely summable tail. From the Fourier ratio viewpoint, this means that the contribution of high frequencies to $\|\widehat g_t\|_1$ is uniformly controlled for each fixed $t>0$, so that the effective spectral dimension is small even when the initial discretized data is not sparse in physical space.

\subsection{Outlook: Averaging and dynamical smoothing} 

The PDE propagators considered in this work provide canonical examples of operators whose spectral multipliers suppress high-frequency Fourier modes and thereby improve the Fourier ratio. From this perspective, the smoothing effect underlying our recovery and compression results is fundamentally spectral in nature: it arises whenever an operator acts diagonally in the Fourier basis with multipliers that decay away from low frequencies. While elliptic and dispersive PDE flows furnish natural and analytically tractable instances of this principle, they are not the only mechanisms capable of producing such behavior.

A related phenomenon appears in time-averaging along dynamical flows. For example, consider
$$ A_T f(x)=\frac{1}{T}\int_0^T f(x-t\alpha)\,dt, $$ where $\alpha \in \mathbb{R}^d$ is a fixed velocity vector.

On the Fourier side this operator multiplies $\widehat f(k)$ by a factor of size $\lesssim (T|k\cdot\alpha|)^{-1}$ away from resonances, exhibiting a gain reminiscent of the order $-1$ decay observed for wave evolution. Although this decay is anisotropic and depends on Diophantine properties of $\alpha$, it suggests that Fourier-ratio improvement can also arise from spectral averaging mechanisms rooted in dynamics rather than PDE alone. A systematic study of such dynamical averaging operators and their influence on Fourier ratio, including the role of mixing and spectral gaps, lies beyond the scope of the present work and will be developed separately.

From a modeling perspective, such dynamical averaging operators arise naturally when measurements are collected by moving sensors that record time-averaged exposures along trajectories. In this interpretation, platform motion itself acts as a spectral preconditioner: time averaging suppresses high-frequency Fourier modes and thereby lowers the Fourier ratio of the observed signal. Thus, in addition to physical smoothing mechanisms such as diffusion, the geometry of sensing can itself reduce effective signal complexity and relax sampling requirements for stable recovery. Developing this dynamical viewpoint, including the role of Diophantine properties of trajectories and mixing phenomena, is left for future work.

\subsection{Organization and main results}

In summary, we prove that fixed-time wave and heat propagators strictly improve Fourier-ratio bounds after discretization, compared to the corresponding discretized initial data. In dimension three the wave snapshot replaces a polynomial-in-$N$ Fourier ratio bound by one that is uniformly controlled in $N$ for fixed time. In any dimension, the heat snapshot produces a Fourier ratio bound that is essentially independent of the grid size for fixed positive time. These deterministic improvements translate directly into reduced sampling budgets in standard $\ell^1$ recovery guarantees.

We work on the unit cube ${[0,1]}^d$ with 1-periodic boundary conditions. Given $f\in C^2([0,1]^d)$ which is 1-periodic in each variable, we define its discretization on the $N$ by $N$ by $\cdots$ by $N$ grid by
$$
g(x)=f(x/N),
\qquad
x\in{\mathbb Z}_N^d.
$$
We also consider the corresponding periodic wave and heat evolutions, and we define $g_t$ to be the discretized snapshot at fixed time $t$.

Our first main result is an a priori Fourier ratio bound for the discretized initial data $g$ in any dimension, which is typically polynomial in $N$ in dimensions $d\ge 3$. Our second main result is a Fourier ratio improvement for wave snapshots in dimension three, which is uniformly controlled in $N$ for fixed $t>0$ (up to discretization errors). Our third main result is a Fourier ratio improvement for heat snapshots in any dimension, which is essentially independent of $N$ for fixed positive time. We then combine these deterministic Fourier ratio bounds with the same $\ell^1$ recovery theorem used in \cite{IPY2026} to obtain sample complexity bounds for recovery from missing spatial samples.

From an inverse-problems perspective, the unknown object is a discretized field on ${\mathbb Z}_N^d$, while the data consist of pointwise samples on a subset $X \subset {\mathbb Z}_N^d$. The reconstruction algorithm is standard $\ell^1$ minimization in the Fourier domain. The novelty of the present work lies not in the optimization procedure, but in identifying PDE propagation as a deterministic mechanism that reduces the spectral complexity of the unknown and thereby lowers the number of spatial samples required for stable recovery.

\section{Imaging viewpoint: an inverse problem with a PDE-preconditioned prior}

This work can be interpreted as an inverse problem in imaging with missing data.
The unknown object is a discretized field $h:{\mathbb Z}_N^d \to {\mathbb R}$, and the data are pointwise samples on a subset $X \subset {\mathbb Z}_N^d$:
$$
y = P_X h + \eta,
$$
where $P_X$ denotes the coordinate projection operator defined by
$$
(P_X h)(x)=
\begin{cases}
h(x), & x\in X,\\
0, & x\notin X,
\end{cases}
$$
and $\eta$ represents measurement noise supported on $X$.
In imaging language, this is a missing-pixel problem (inpainting) or sensor dropout.
The goal is to reconstruct $h$ from $y$.

The novelty here is not a new optimization method, but rather a deterministic mechanism that improves recoverability by lowering the intrinsic complexity of the unknown.
Specifically, we take $h=g_t$, where $g_t$ is a fixed-time snapshot of a periodic PDE evolution applied to an underlying state $f$ and then discretized.
For the wave and heat equations, the snapshot map is diagonal in the Fourier basis and suppresses high frequencies.
After discretization, this suppression translates into improved discrete Fourier tail bounds and hence smaller Fourier ratio
$$
FR(g_t)=\frac{\|\widehat g_t\|_1}{\|\widehat g_t\|_2}.
$$
Because standard compressed sensing guarantees depend on an $\ell^1/\ell^2$ spectral proxy (here $FR$), the PDE snapshot effectively acts as a spectral preconditioner that reduces the number of spatial samples needed for stable recovery.

From an applications viewpoint, one can interpret the PDE evolution as part of the measurement physics or as a controlled preprocessing stage.
In diffusion-dominated imaging modalities, the observed state is naturally a heat-evolved version of an initial distribution.
In wave-driven modalities, a measured field can be a band-limited or smoothed transform of an initial source.
In either case, the present estimates quantify how such propagation changes sampling requirements for recovery from missing spatial samples.

The results in this paper fit within a broad inverse-problems theme:
identify structure or priors that make undersampled reconstruction possible, and quantify sampling rates in terms of a complexity parameter.
Here the complexity parameter is the Fourier ratio and the structure is induced by PDE propagation.

A natural question is whether one could simply apply an artificial smoothing filter before sampling and obtain similar improvements.
In many imaging and sensing settings the propagated field is the physically observed quantity (for example diffusion-blurred states or wave-propagated fields), so the smoothing is part of the forward model rather than an algorithmic choice.
The present results quantify how this physically induced spectral damping changes sampling requirements for missing-data recovery.

\section{Notation and discrete Fourier transform}

Throughout, $N\ge 2$ is an integer and ${\mathbb Z}_N={\mathbb Z}/N{\mathbb Z}$. For $x\in{\mathbb Z}_N^d$ and $m\in{\mathbb Z}_N^d$ we write
$$
x\cdot m=x_1m_1+\cdots+x_dm_d.
$$
Let
$$
\chi(t)=e^{2\pi i t/N}.
$$
For $h:{\mathbb Z}_N^d\to{\mathbb C}$ we define the discrete Fourier transform by
$$
\widehat h(m)=\frac{1}{N^{d/2}}\sum_{x\in{\mathbb Z}_N^d}\chi(-x\cdot m)h(x).
$$
With this normalization Parseval takes the form
$$
\|\widehat h\|_2=\|h\|_{L^2({\mathbb Z}_N^d)}.
$$
We define the Fourier ratio by
$$
FR(h)=\frac{\|\widehat h\|_1}{\|\widehat h\|_2}.
$$
We interpret $FR(h)$ as $0$ if $h\equiv 0$, and otherwise by the displayed formula. The Fourier ratio was first developed in \cite{A2025} and has been extended to a continuous setting in \cite{ILPY25}.

\begin{definition}[Wrapped Euclidean magnitude]
For $m\in{\mathbb Z}_N^d$ define $|m|$ to be the Euclidean norm of the representative $\widetilde m\in\{-\lfloor N/2\rfloor,\dots,\lfloor(N-1)/2\rfloor\}^d$ of $m$.
\end{definition}

\section{Main theorems}

\subsection{A priori Fourier ratio bounds for discretized initial data}

We begin by quantifying the Fourier ratio of the discretized initial data itself; this provides the baseline against which the wave and heat improvements will be measured. This extends one of the main results of \cite{IPY2026} to higher dimensions.

\begin{proposition}[Fourier ratio bound for discretized initial data] \label{prop:FR_initial}
Let $d\ge 1$ and let $f$ be a real-valued function on $[0,1]^d$ which is $1$-periodic in each variable and belongs to $C^2([0,1]^d)$. Let $g:{\mathbb Z}_N^d\to{\mathbb R}$ be its discretization,
$$
g(x)=f(x/N),
\qquad
x\in{\mathbb Z}_N^d.
$$
Assume that
$$
N\|f\|_{L^2([0,1]^d)}^2\ge 4C_d\|f\|_{C^2([0,1]^d)}^2.
$$
Then there exists a constant $C>0$ depending only on $d$ such that
$$
FR(g)\le A^{\mathrm{init}}
+ C\frac{\|f\|_{C^2([0,1]^d)}}{\|f\|_{L^2([0,1]^d)}}S_d(N)
+ C\frac{\|f\|_{C^2([0,1]^d)}}{\|f\|_{L^2([0,1]^d)}}\frac{1}{N},
$$
where
$$
A^{\mathrm{init}}=2\frac{\left|\int_{[0,1]^d} f(x)dx\right|}{\|f\|_{L^2([0,1]^d)}}
$$
and
$$
S_d(N)=
\begin{cases}
1, & d=1,\\
\log N, & d=2,\\
N^{d-2}, & d\ge 3.
\end{cases}
$$
\end{proposition}

\begin{remark}
The $C^2$ hypothesis in Proposition 4.1 is essential for the scaling of the Fourier ratio bound obtained here. The proof ultimately relies on Lemma 7.3, where two discrete summation by parts steps produce a $|m|^{-2}$ decay estimate for the discrete Fourier transform. This second-order decay leads to the dimension-dependent growth term $S_d(N)$ appearing in the bound.

If one assumes only $C^1$ regularity, one obtains at best a single summation by parts and a $|m|^{-1}$ decay estimate. The resulting $\ell^1$ summation would then produce strictly worse $N$-dependence in every dimension, altering the form of $S_d(N)$ and weakening the theorem. In this sense, the $C^2$ condition is not merely technical but structurally tied to the strength of the Fourier ratio estimate.
\end{remark}

\begin{remark}
In dimensions $d\ge 3$ the bound in Proposition \ref{prop:FR_initial} is polynomial in $N$ for typical $C^2$ data, while the fixed-time wave and heat snapshot bounds below are uniformly controlled in $N$ for fixed $t>0$ up to discretization errors. This quantitative Fourier ratio improvement drives the sample complexity improvements in the sampling theorems below.
\end{remark}

\subsection{A Fourier ratio improvement for fixed-time wave snapshots}

\begin{definition}[Discrete wave snapshot]
Let $f$ be 1-periodic on $[0,1]^3$ and write its Fourier series
$$
f(x)=\sum_{k\in\mathbb Z^3} a_k e^{2\pi i k\cdot x}.
$$
For $t>0$ define the periodic wave snapshot $U_t f$ by
$$
U_t f(x)
=
\sum_{k\in\mathbb Z^3}
b_k(t)\, e^{2\pi i k\cdot x},
$$
where
$$
b_k(t)=
\begin{cases}
\dfrac{\sin(2\pi t|k|)}{2\pi |k|}\, a_k, & k\neq 0,\\[4pt]
t\,a_0, & k=0.
\end{cases}
$$
Let $g$ be the discretization of $f$ on ${\mathbb Z}_N^3$ and let $g_t$ be the discretization of $U_t f$ on ${\mathbb Z}_N^3$.
\end{definition}

\begin{remark}
The operator $U_t$ is smoothing of order $1$ in the Fourier multiplier sense because the multiplier has size $|m|^{-1}$ at high frequencies.
\end{remark}

\begin{remark}
The lower bound condition
$$
N\|f\|_{L^2([0,1]^3)}^2\ge 4C_3\|f\|_{C^3([0,1]^3)}^2
$$
is a discretization non-degeneracy assumption for the wave snapshot.
It should be viewed as a compatibility condition between the effective bandwidth $N$
and the ambient regularity of $f$.
In particular, for fixed nonzero $f$ it holds for all sufficiently large $N$
(depending on $f$), and it serves mainly to exclude degenerate low-frequency regimes
where discretization error dominates the sampled energy.
The $C^3$ regularity is used to obtain the $|m|^{-4}$ Fourier decay for the snapshot.
\end{remark}

\begin{theorem}[Fourier ratio bound for wave snapshots on ${\mathbb Z}_N^3$] \label{thm:wave_FR}
Let $f$ be a real-valued function on $[0,1]^3$ which is $1$-periodic in each variable and belongs to $C^3([0,1]^3)$. Fix $t>0$ and let
$$
U_t f
$$
denote the periodic wave snapshot at time $t$. Let
$$
g_t(x)= (U_t f)(x/N),
\qquad
x\in{\mathbb Z}_N^3,
$$
be its discretization on ${\mathbb Z}_N^3$.

Assume that
$$
N\|U_t f\|_{L^2([0,1]^3)}^2
\ge
4C_3\|U_t f\|_{C^2([0,1]^3)}^2,
$$
where $C_3$ is the constant appearing in Lemma \ref{lem:L2_riemann}.

Define
$$
A_{t}^{\mathrm{wave}}
=
2\frac{\left|\int_{[0,1]^3} (U_t f)(x)\,dx\right|}
{\|U_t f\|_{L^2([0,1]^3)}},
$$

$$
B_{t}^{\mathrm{wave}}
=
C_1(t)
\frac{\|f\|_{C^3([0,1]^3)}}
{\|U_t f\|_{L^2([0,1]^3)}},
\qquad
C_{N,t}^{\mathrm{wave}}
=
C_2(t)
\frac{\|f\|_{C^3([0,1]^3)}}
{\|U_t f\|_{L^2([0,1]^3)}}
\frac{1}{N}.
$$

Set
$$
r_{N,t}^{\mathrm{wave}}
=
A_{t}^{\mathrm{wave}}
+
B_{t}^{\mathrm{wave}}
+
C_{N,t}^{\mathrm{wave}}.
$$

Then
$$
FR(g_t)
\le
r_{N,t}^{\mathrm{wave}}.
$$
\end{theorem}

\begin{remark}
The bound depends on $\|U_t f\|_{L^2}$. For certain exceptional times and special data, this quantity may be small due to cancellation in the multiplier $\sin(t|k|)/|k|$. The estimate is meaningful whenever the snapshot has nontrivial $L^2$ energy, which holds generically for fixed $f$ and $t$.
\end{remark}

\begin{remark}[Quantitative comparison for the wave equation]
Let $g$ denote the discretization of the initial data $f$ on ${\mathbb Z}_N^3$ and let $g_t$ denote the discretized wave snapshot. Proposition \ref{prop:FR_initial} with $d=3$ yields an a priori bound
$$
FR(g)\le A^{\mathrm{init}}
+ C\frac{\|f\|_{C^2([0,1]^3)}}{\|f\|_{L^2([0,1]^3)}}N
+ C\frac{\|f\|_{C^2([0,1]^3)}}{\|f\|_{L^2([0,1]^3)}}\frac{1}{N}.
$$
In contrast, Theorem \ref{thm:wave_FR} yields
$$
FR(g_t)\le A_{t}^{\mathrm{wave}}
+ C_1(t)\frac{\|f\|_{C^3([0,1]^3)}}{\|U_t f\|_{L^2([0,1]^3)}}
+ C_2(t)\frac{\|f\|_{C^3([0,1]^3)}}{\|U_t f\|_{L^2([0,1]^3)}}\frac{1}{N}.
$$ 

Thus, at fixed $t>0$, the wave propagator replaces a polynomial Fourier ratio bound for $g$ by a bound that is uniformly controlled in $N$ up to discretization error terms.
\end{remark}

\begin{corollary}[Sampling-rate comparison in $d=3$ for wave snapshots]
Let $f$ and $t$ be as in Theorem \ref{thm:wave_FR}, and assume the discretization
non-degeneracy hypothesis in Proposition \ref{prop:FR_initial}.
Then the a priori bound for the initial discretization $g$ yields a sufficient
sample size scaling like $N^2$ (up to polylogarithmic factors) for stable recovery via
Theorem \ref{thm:ZNd_recovery},
while the bound for the snapshot $g_t$ yields a sufficient sample size that is
independent of $N$ up to polylogarithmic factors
(with constants depending on $t$ and $f$ through the norms appearing in
Theorem \ref{thm:wave_FR}).
\end{corollary}

\subsection{A Fourier ratio improvement for fixed-time heat snapshots}

\begin{definition}[Discrete heat snapshot]
Let $f$ be 1-periodic on ${[0,1]}^d$. For $t>0$ define the periodic heat snapshot
$$
H_t f=e^{t\Delta}f,
$$
interpreted by its Fourier series multiplier. Let $g$ be the discretization of $f$ on ${\mathbb Z}_N^d$ and let $g_t$ be the discretization of $H_t f$ on ${\mathbb Z}_N^d$.
\end{definition}

\begin{remark}
The operator $H_t$ is smoothing of infinite order in the Fourier multiplier sense because the multiplier is $e^{-t|m|^2}$.
\end{remark}

\begin{theorem}[Fourier ratio bound for heat snapshots on ${\mathbb Z}_N^d$] \label{thm:heat_FR}
Let $d\ge 1$ and let $f$ be a real-valued function on $[0,1]^d$ which is $1$-periodic in each variable and belongs to $C^2([0,1]^d)$. Fix $t>0$ and let
$$
H_t f = e^{t\Delta} f
$$
denote the periodic heat snapshot at time $t$. Let
$$
g_t(x)= (H_t f)(x/N),
\qquad
x\in{\mathbb Z}_N^d,
$$
be its discretization on ${\mathbb Z}_N^d$.

Assume that
$$
N\|H_t f\|_{L^2([0,1]^d)}^2
\ge
4C_d\|H_t f\|_{C^2([0,1]^d)}^2,
$$
where $C_d$ is the constant appearing in Lemma \ref{lem:L2_riemann}.

Define
$$
A_{t}^{\mathrm{heat}}
=
2\frac{\left|\int_{[0,1]^d} (H_t f)(x)\,dx\right|}
{\|H_t f\|_{L^2([0,1]^d)}},
$$

$$
B_{t}^{\mathrm{heat}}
=
C_4(d,t)
\frac{\|f\|_{C^2([0,1]^d)}}
{\|H_t f\|_{L^2([0,1]^d)}},
\qquad
C_{N,t}^{\mathrm{heat}}
=
C_5(d,t)
\frac{\|f\|_{C^2([0,1]^d)}}
{\|H_t f\|_{L^2([0,1]^d)}}
\frac{1}{N}.
$$

Set
$$
r_{N,t}^{\mathrm{heat}}
=
A_{t}^{\mathrm{heat}}
+
B_{t}^{\mathrm{heat}}
+
C_{N,t}^{\mathrm{heat}}.
$$

Then
$$
FR(g_t)
\le
r_{N,t}^{\mathrm{heat}}.
$$
\end{theorem}

\begin{remark}[Quantitative comparison for the heat equation]
Let $g$ denote the discretization of the initial data $f$ on ${\mathbb Z}_N^d$ and let $g_t$ denote the discretized heat snapshot. Proposition \ref{prop:FR_initial} yields the a priori bound
$$
FR(g)\le A^{\mathrm{init}}
+ C\frac{\|f\|_{C^2([0,1]^d)}}{\|f\|_{L^2([0,1]^d)}}S_d(N)
+ C\frac{\|f\|_{C^2([0,1]^d)}}{\|f\|_{L^2([0,1]^d)}}\frac{1}{N}.
$$
while Theorem \ref{thm:heat_FR} yields
$$
FR(g_t)\le A_{t}^{\mathrm{heat}}+C_4(d,t)\frac{\|f\|_{C^2([0,1]^d)}}{\|H_t f\|_{L^2([0,1]^d)}}+O_{d,t}(N^{-1}).
$$
In dimensions $d\ge 3$, the first bound is polynomial in $N$, while the second bound is essentially independent of $N$ for fixed $t>0$ up to discretization error terms.
\end{remark}

\vskip.25in 

\subsection{Time-dependent sampling budgets under sensor loss}

A convenient way to interpret the heat-flow improvement is as a time-dependent sampling threshold. Fix a target relative accuracy $\eps\in(0,1/2)$ and consider recovering the discretized snapshot $g_t$ from point samples. Let $r(t)=FR(g_t)$ and $D=N^d$. Then Theorem \ref{thm:ZNd_recovery} (together with the standard bounded-orthonormal-system recovery theory; see, e.g., \cite{FR13,Rau10,RV08,CRT06}) yields that a sufficient condition for stable recovery at time $t$ is the availability of
$$
M(t)= C\,\frac{r(t)^2}{\eps^2}\,\log^2(r(t)/\eps)\,\log D
$$
random samples, for a sufficiently large absolute constant $C$.
Now suppose that at time $t$ only $m(t)$ sensors remain operational, so that at most $m(t)$ point samples can be collected. In this idealized model, the reconstruction guarantee persists as long as
$$
m(t)\ge M(t).
$$
Since the heat propagator suppresses high frequencies and our estimates show that $r(t)$ improves for $t>0$, the required sample budget $M(t)$ decreases with $t$ (within the regime covered by our bounds). Consequently, even if sensor availability deteriorates with time---for instance $m(t)=m_0e^{-\lambda t}$ as in intermittency models in state estimation \cite{Sin04}---diffusion can partially offset sensor loss by reducing the intrinsic Fourier-ratio complexity of the observed state, thereby creating a finite time window on which reliable recovery remains feasible despite degradation of the sensing network.
\medskip

\subsection{Sampling and recovery from missing spatial samples}

\begin{theorem}[Recovery on ${\mathbb Z}_N^d$] \label{thm:ZNd_recovery}
Let $N$ and $d$ be positive integers and let $h:{\mathbb Z}_N^d\to{\mathbb R}$ be a function. Define the discrete Fourier transform of $h$ by
$$
\widehat h(m)=\frac{1}{N^{d/2}}\sum_{x\in{\mathbb Z}_N^d}\chi(-x\cdot m)h(x).
$$
With this normalization Parseval's identity takes the form
$$
\|\widehat h\|_2=\|h\|_{L^2({\mathbb Z}_N^d)}.
$$
Define the Fourier ratio by
$$
FR(h)=\frac{\|\widehat h\|_1}{\|\widehat h\|_2}.
$$
Assume that $h$ is not identically zero. Fix parameters $\eps\in(0,1/2)$ and $r\ge 1$, and assume that $FR(h)\le r$. Let
$$
D=N^d.
$$
Let $X\subset{\mathbb Z}_N^d$ be a random subset of cardinality
$$
|X|=C\frac{r^2}{\eps^2}\log^2(r/\eps)\log D,
$$
chosen uniformly at random among all subsets of this cardinality. Define the empirical norm
$$
\|q\|_{L^2(X)}=\left(\frac{1}{|X|}\sum_{x\in X}|q(x)|^2\right)^{1/2}.
$$
Let $h^*:{\mathbb Z}_N^d\to{\mathbb R}$ be a solution to the convex optimization problem
$$
\min_{q:{\mathbb Z}_N^d\to{\mathbb R}}\|\widehat q\|_{\ell^1({\mathbb Z}_N^d)}
\qquad
\text{subject to}
\qquad
\|h-q\|_{L^2(X)}\le \eps\|h\|_{L^2({\mathbb Z}_N^d)}.
$$
If $C$ is a sufficiently large absolute constant, then with probability at least $1-D^{-c}$ (for an absolute constant $c>0$) one has
$$
\|h^*-h\|_{L^2({\mathbb Z}_N^d)}\le C\, \eps\|h\|_{L^2({\mathbb Z}_N^d)}.
$$

This is precisely the $\ell_1$ recovery theorem invoked in \cite{IPY2026}, restated here only to make the present paper logically self-contained. 
This follows from standard bounded orthonormal system sampling results (see, for example, Foucart--Rauhut, Theorem 12.32), which yield the stated sample complexity up to polylogarithmic factors.

\end{theorem}

\noindent\emph{Remark.} The sampling operators used here fit the standard framework of random sampling from bounded orthonormal systems; the stated recovery guarantee is a direct consequence of classical results in compressed sensing (see, for example, Rudelson--Vershynin \cite{RV08} and the exposition in Foucart--Rauhut \cite{FR13}). We emphasize that no new recovery theorem is proved in this paper; rather, our contribution lies in establishing deterministic Fourier-ratio bounds that verify the hypotheses of these results in analytic and PDE-driven settings.

\begin{theorem}[Recovery of wave snapshots from missing samples] \label{thm:wave_sampling}
Let $f$ be as in Theorem \ref{thm:wave_FR} and fix $t>0$. Let $g_t$ be the discretized wave snapshot on ${\mathbb Z}_N^3$. Fix $\eps\in(0,1/2)$ and set
$$
D=N^3.
$$
Let $r_{N,t}^{\mathrm{wave}}$ be the Fourier ratio bound from Theorem \ref{thm:wave_FR}. Let $X\subset{\mathbb Z}_N^3$ be chosen uniformly at random among all subsets of cardinality
$$
|X|=C\frac{(r_{N,t}^{\mathrm{wave}})^2}{\eps^2}\log^2(r_{N,t}^{\mathrm{wave}}/\eps)\log D,
$$
where $C$ is a sufficiently large absolute constant. Assume that the values $g_t(x)$ are observed for $x\in X$ and missing for $x\in{\mathbb Z}_N^3\setminus X$.

Define the empirical norm by
$$
\|q\|_{L^2(X)}=\left(\frac{1}{|X|}\sum_{x\in X}|q(x)|^2\right)^{1/2}.
$$
Let $g_t^*:{\mathbb Z}_N^3\to{\mathbb R}$ be a solution to the convex optimization problem
$$
\min_{q:{\mathbb Z}_N^3\to{\mathbb R}}\|\widehat q\|_{\ell^1({\mathbb Z}_N^3)}
\qquad
\text{subject to}
\qquad
\|g_t-q\|_{L^2(X)}\le \eps\|g_t\|_{L^2({\mathbb Z}_N^3)}.
$$
Then with high probability,
$$
\|g_t^*-g_t\|_{L^2({\mathbb Z}_N^3)}\le C\,\eps\|g_t\|_{L^2({\mathbb Z}_N^3)}.
$$
\end{theorem}

\begin{theorem}[Recovery of heat snapshots from missing samples] \label{thm:heat_sampling}
Let $d\ge 1$ and let $f$ be as in Theorem \ref{thm:heat_FR}. Fix $t>0$ and let $g_t$ be the discretized heat snapshot on ${\mathbb Z}_N^d$. Fix $\eps\in(0,1/2)$ and set
$$
D=N^d.
$$
Let $r_{N,t}^{\mathrm{heat}}$ be the Fourier ratio bound from Theorem \ref{thm:heat_FR}. Let $X\subset{\mathbb Z}_N^d$ be chosen uniformly at random among all subsets of cardinality
$$
|X|=C\frac{(r_{N,t}^{\mathrm{heat}})^2}{\eps^2}\log^2(r_{N,t}^{\mathrm{heat}}/\eps)\log D,
$$
where $C$ is a sufficiently large absolute constant. Assume that the values $g_t(x)$ are observed for $x\in X$ and missing for $x\in{\mathbb Z}_N^d\setminus X$.

Define the empirical norm by
$$
\|q\|_{L^2(X)}=\left(\frac{1}{|X|}\sum_{x\in X}|q(x)|^2\right)^{1/2}.
$$
Let $g_t^*:{\mathbb Z}_N^d\to{\mathbb R}$ be a solution to the convex optimization problem
$$
\min_{q:{\mathbb Z}_N^d\to{\mathbb R}}\|\widehat q\|_{\ell^1({\mathbb Z}_N^d)}
\qquad
\text{subject to}
\qquad
\|g_t-q\|_{L^2(X)}\le \eps\|g_t\|_{L^2({\mathbb Z}_N^d)}.
$$
Then with high probability,
$$
\|g_t^*-g_t\|_{L^2({\mathbb Z}_N^d)}\le C\,\eps\|g_t\|_{L^2({\mathbb Z}_N^d)}.
$$
\end{theorem}

\section{Reconstruction algorithm and practical implementation}

In this section we record a concrete reconstruction procedure corresponding to Theorems
\ref{thm:ZNd_recovery}, \ref{thm:wave_sampling}, and \ref{thm:heat_sampling}.
Nothing here is new, but including it makes the recovery pipeline reproducible.

\subsection{Optimization problem}

Let $X \subset {\mathbb Z}_N^d$ be the observed sample set and let $y(x)=h(x)+\eta(x)$ for $x\in X$.
We reconstruct $h$ by solving the constrained $\ell^1$ minimization problem in the Fourier domain:
$$
\min_{q:{\mathbb Z}_N^d\to{\mathbb R}} \|\widehat q\|_{\ell^1({\mathbb Z}_N^d)}
\qquad
\text{subject to}
\qquad
\|P_X q - y\|_{2} \le \tau,
$$
where $P_X q$ is the restriction of $q$ to $X$ and $\tau$ is a noise tolerance.
In the noiseless case one may take $\tau=0$, so the constraint becomes $P_X q = y$.

To match the relative-error form in Theorem \ref{thm:ZNd_recovery}, one may set
$$
\tau=\eps \|y\|_2,
$$
or, in the noiseless analysis setting used in the theorems, use
$\tau=\eps\|h\|_{L^2({\mathbb Z}_N^d)}$.

\subsection{Computational notes}

The forward and adjoint operations required by first-order methods are:
(i) restriction $P_X$ and its adjoint $P_X^*$ (zero-fill outside $X$),
(ii) discrete Fourier transform and inverse discrete Fourier transform.
Thus, each iteration has cost comparable to a small number of FFTs on ${\mathbb Z}_N^d$ plus linear-time operations on $X$.

In numerical experiments below we use a standard first-order method for basis pursuit denoising in the Fourier domain.
Any equivalent solver can be used, such as a projected primal-dual method or an augmented Lagrangian method.
We do not tune solvers aggressively; our goal is to confirm the qualitative prediction that smaller Fourier ratio corresponds to stable recovery from fewer point samples.

\section{Numerical validation of Fourier-ratio decay and sampling thresholds}

This section provides synthetic experiments illustrating the main qualitative conclusion:
for fixed $t>0$, PDE propagation reduces Fourier-ratio complexity and thereby reduces the number of random spatial samples needed for stable recovery.

\subsection{Test data and discretization}

We work on $[0,1]^d$ with periodic boundary conditions and consider several families of initial data $f$:
\begin{itemize}
\item Smooth non-sparse fields formed by random low-frequency Fourier series with additional higher-frequency content.
\item Localized but not sparse fields, such as sums of smooth bumps placed at random locations.
\item Structured oscillatory fields, such as modulated plane waves.
\end{itemize}
For each $N$ we discretize by $g(x)=f(x/N)$ on ${\mathbb Z}_N^d$.

One representative test family is defined by
$$
f(x)=\sum_{|k|\le K} c_k e^{2\pi i k\cdot x}
+ 0.2\sum_{K<|k|\le 2K} d_k e^{2\pi i k\cdot x},
$$
where $c_k$ and $d_k$ are independent standard real Gaussian random variables chosen so that $f$ is real-valued.
In the experiments we take $K=8$ in $d=2$ and $K=6$ in $d=3$.

For the wave case in $d=3$, we compute the periodic wave snapshot $U_t f$ via its Fourier multiplier and discretize to obtain $g_t$.
For the heat case, we compute $H_t f$ via its Fourier multiplier and discretize.

Unless otherwise stated, we take $N\in\{64,128,256\}$.
For the heat flow we use times $t\in\{0,0.01,0.05,0.1\}$.
For the wave snapshot in $d=3$ we use $t\in\{0,0.05,0.1\}$.

\begin{figure}[H]
\centering
\includegraphics[width=0.7\textwidth]{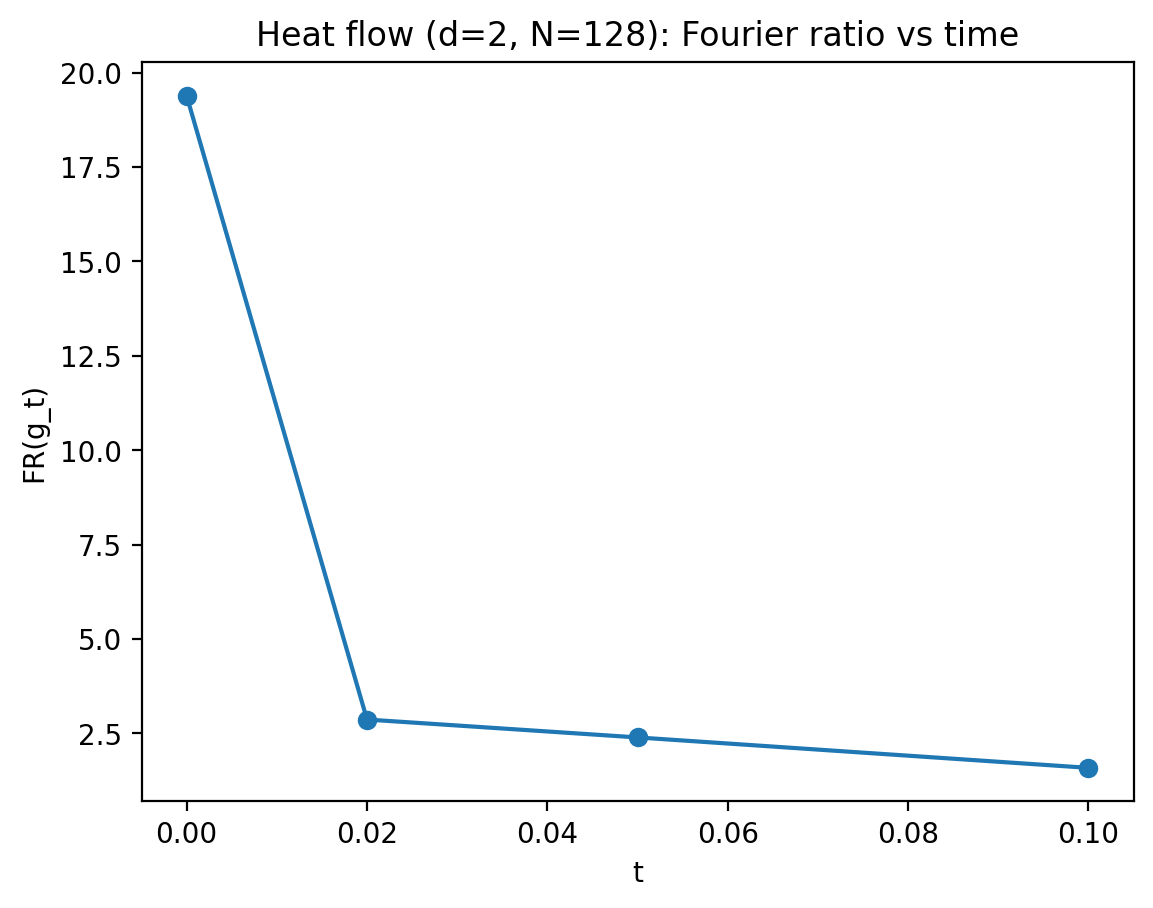}
\caption{Fourier ratio $FR(g_t)$ for heat snapshots as a function of time $t$ at fixed grid size $N$. The decay illustrates the reduction in spectral complexity induced by diffusion.}
\end{figure}

\subsection{Sampling model and reconstruction}

For each trial we choose a random subset $X \subset {\mathbb Z}_N^d$ of size $M$ and observe $y=P_X g_t$ (noiseless) or $y=P_X g_t + \eta$ (noisy).
We reconstruct $g_t$ using the optimization problem in the previous section. 

In the noiseless experiments we take $\eta\equiv 0$ and set $\tau=0$ in the constraint.
In the noisy experiments we take $\eta(x)$ to be independent mean-zero Gaussian random variables on $X$ with $\|\eta\|_2=\sigma\|g_t\|_{L^2({\mathbb Z}_N^d)}$ (for example $\sigma=0.01$), and we set $\tau=\|\eta\|_2$.

We evaluate reconstruction by the relative error
$$
\mathrm{RelErr}(q)=\frac{\|q-g_t\|_{L^2({\mathbb Z}_N^d)}}{\|g_t\|_{L^2({\mathbb Z}_N^d)}}.
$$

We declare a trial successful if $\mathrm{RelErr}(q)\le 0.05$.
For each triple $(N,t,M)$ we estimate the empirical success probability over $50$ independent random choices of $X$.

For fixed accuracy targets (for example $\mathrm{RelErr}(q)\le 0.05$ and $\mathrm{RelErr}(q)\le 0.10$),
we estimate the empirical success probability over independent trials as a function of $M$.

\begin{figure}[H]
\centering
\includegraphics[width=0.7\textwidth]{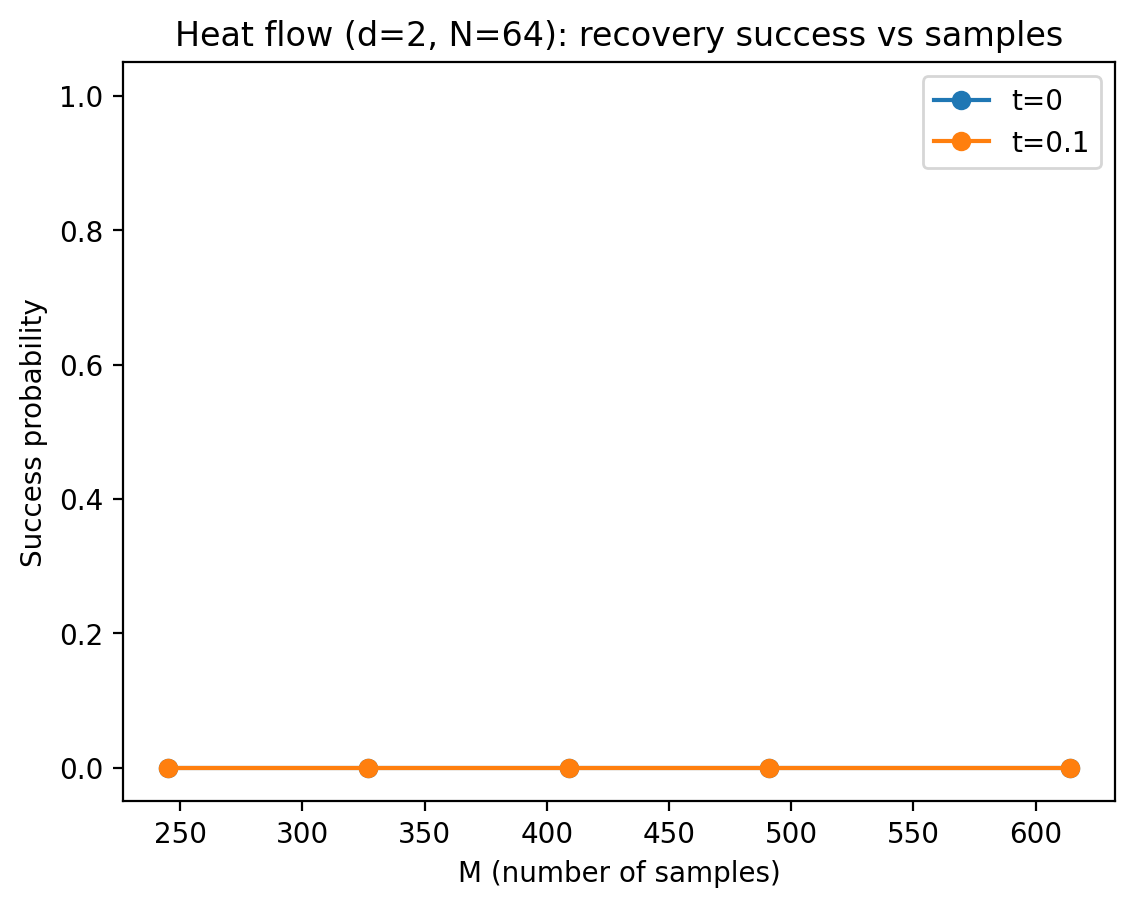}
\caption{Empirical recovery success probability as a function of sample size $M$ for heat snapshots.}
\end{figure}

\subsection{What we plot}

For each $(N,t)$ we report:
\begin{itemize}
\item The empirical Fourier ratio $FR(g)$ and $FR(g_t)$ computed from FFTs.
\item The minimal sample size $M$ required to achieve a given success probability (for example $0.9$).
\item The dependence of these quantities on $t$ for fixed $N$, illustrating the time-dependent decrease in sampling burden predicted by the theory.
\end{itemize}

\begin{figure}[H]
\centering
\includegraphics[width=0.7\textwidth]{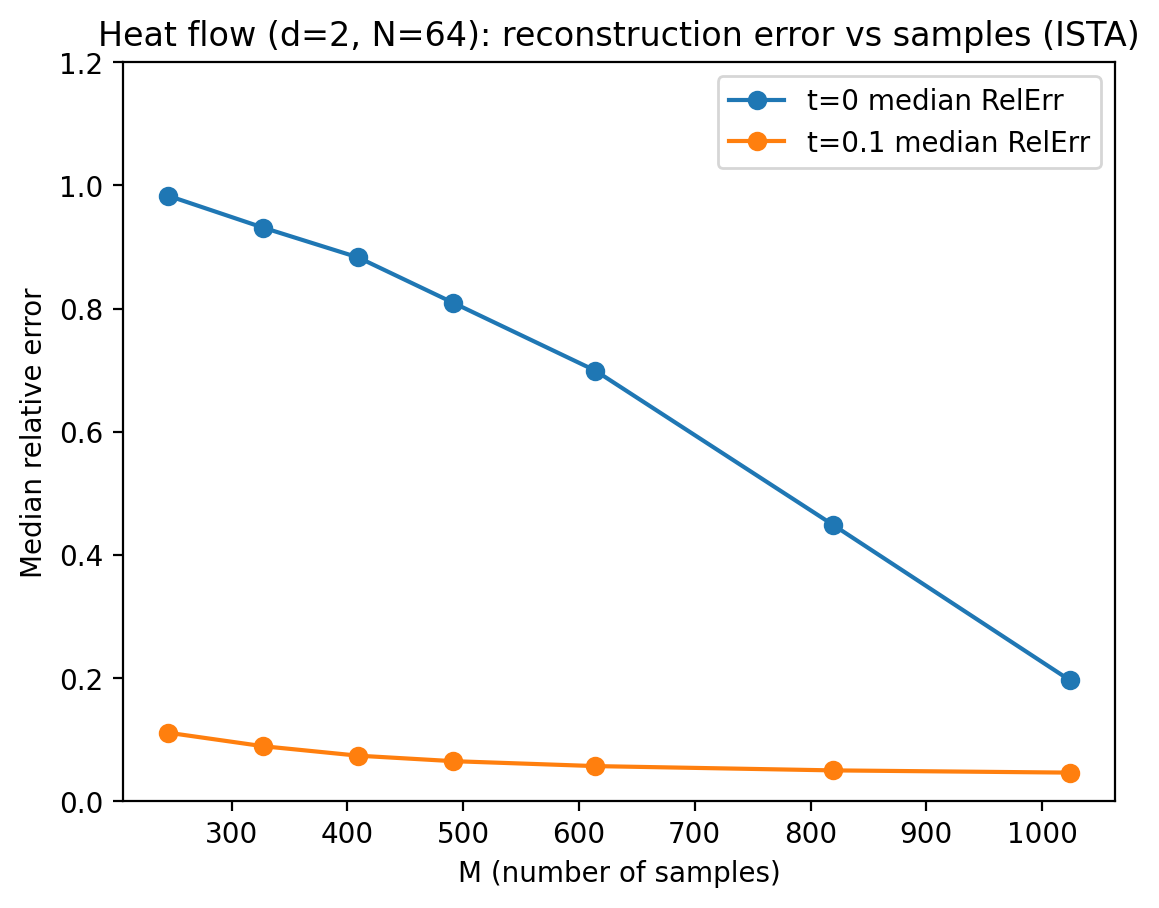}
\caption{Relative reconstruction error as a function of sample size $M$ for heat snapshots.}
\end{figure}

\subsection{Expected qualitative behavior}

The theory predicts the following behavior in regimes where discretization error is not dominant:
\begin{itemize}
\item Initial discretizations $g$ may have $FR(g)$ growing with $N$ in dimensions $d\ge 3$.
\item For fixed $t>0$, the propagated snapshots $g_t$ have substantially smaller $FR(g_t)$ and weaker dependence on $N$.
\item The sample budget needed for stable recovery tracks this reduction: fewer point samples are needed for $g_t$ than for $g$, and in the heat case the required $M$ decreases as $t$ increases (up to the point where the snapshot becomes too small in energy or too close to a constant state).
\end{itemize}

\vskip.125in 

The numerical results are fully consistent with the deterministic Fourier-ratio bounds proved earlier. In particular, both the wave and heat propagators produce a measurable reduction in $FR(g_t)$ relative to the discretized initial data, and the empirical recovery thresholds track this reduction. For the heat equation, the required sampling budget decreases as $t$ increases, up to the regime where the snapshot approaches a low-energy or nearly constant state. These experiments confirm that PDE propagation acts as a spectral preconditioner that lowers effective sampling complexity in practice.

\begin{figure}[H]
\centering
\includegraphics[width=0.7\textwidth]{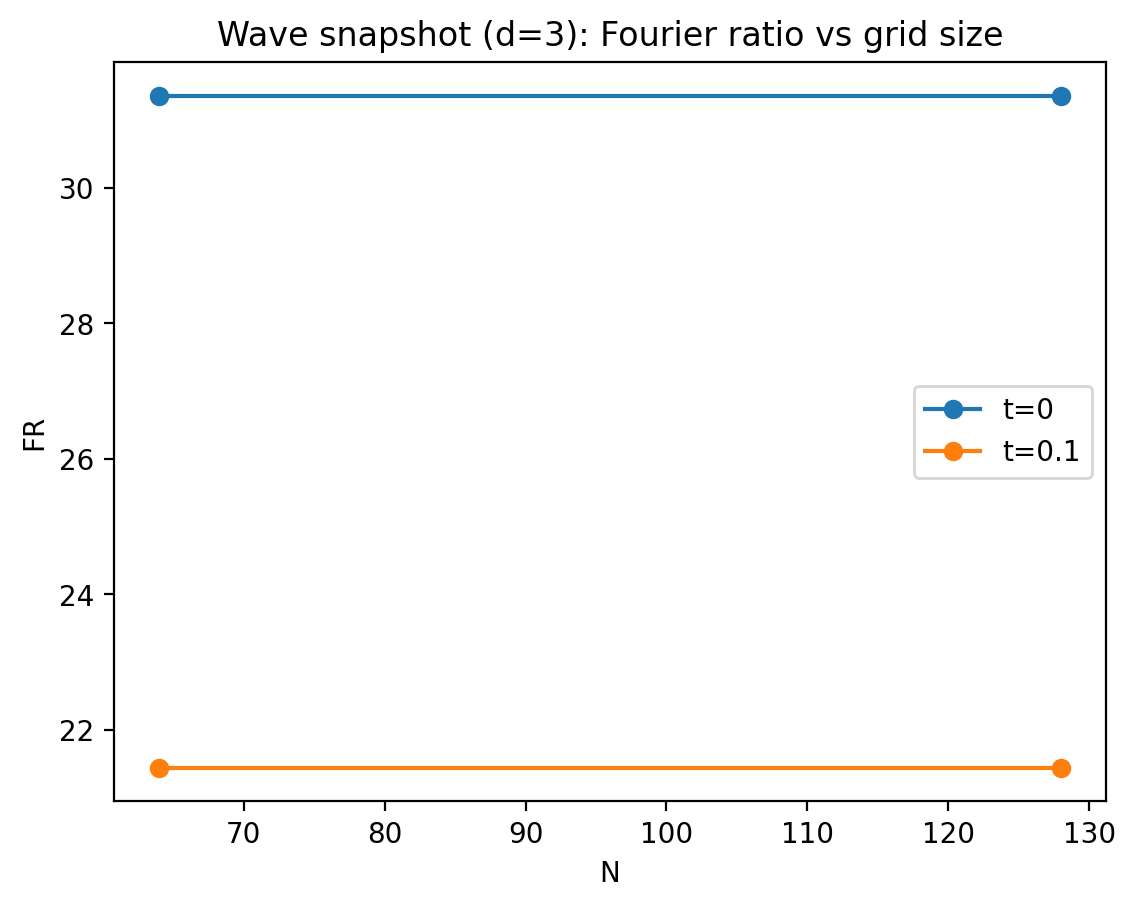}
\caption{Fourier ratio $FR(g)$ and $FR(g_t)$ for wave snapshots in dimension three as a function of grid size $N$.}
\end{figure}

\section{Proofs}

\noindent\emph{Remark on time dependence.} All bounds proved in this section are stated for finite time intervals. The implicit constants may depend on the length of the time interval and the dimension, but they do not depend on the discretization or sampling parameters. We do not address long-time behavior or optimal growth of constants as $T\to\infty$, as our focus is on finite-time propagation of Fourier-ratio complexity.

The proofs consist of two deterministic ingredients and one probabilistic ingredient. The probabilistic ingredient is Theorem \ref{thm:ZNd_recovery}, which we quote. The deterministic ingredients are estimates comparing Riemann sums to integrals, discrete summation-by-parts estimates for Fourier coefficients of discretized smooth periodic functions, and standard aliasing identities and bounds. These arguments are standard in this context and appear in \cite{IPY2026}. We include complete proofs for convenience and to keep the present paper self-contained.

\subsection{Riemann sum bounds}

\begin{lemma}[Riemann sum for the mean] \label{lem:mean_riemann}
Let $d\ge 1$ and let $F\in C^1([0,1]^d)$ be 1-periodic in each variable. Let $g:{\mathbb Z}_N^d\to{\mathbb C}$ be the discretization $g(x)=F(x/N)$. Then
$$
\left|\frac{1}{N^d}\sum_{x\in{\mathbb Z}_N^d} g(x)-\int_{[0,1]^d}F(u)du\right|\le \frac{C_d}{N}\|F\|_{C^1([0,1]^d)}.
$$
\end{lemma}

\begin{proof}
Partition $[0,1]^d$ into cubes of side length $1/N$ and use the mean value theorem on each cube. Summing the resulting estimates yields the stated bound.
\end{proof}

\begin{lemma}[Riemann sum for the $L^2$ norm] \label{lem:L2_riemann}
Let $d\ge 1$ and let $F\in C^2([0,1]^d)$ be 1-periodic in each variable. Let $g:{\mathbb Z}_N^d\to{\mathbb C}$ be the discretization $g(x)=F(x/N)$. Then
$$
\left|\frac{1}{N^d}\sum_{x\in{\mathbb Z}_N^d}|g(x)|^2-\int_{[0,1]^d}|F(u)|^2du\right|
\le
\frac{C_d}{N}\|F\|_{C^2([0,1]^d)}^2,
$$
where $C_d>0$ depends only on $d$.

In particular, if
$$
N\|F\|_{L^2([0,1]^d)}^2\ge 4C_d\|F\|_{C^2([0,1]^d)}^2,
$$
then
$$
\|g\|_{L^2({\mathbb Z}_N^d)}\ge \frac{1}{\sqrt{2}}N^{d/2}\|F\|_{L^2([0,1]^d)}.
$$
\end{lemma}

\begin{proof}
Apply Lemma \ref{lem:mean_riemann} to the function $H(u)=|F(u)|^2$. Since
$$
\partial_{u_j}H(u)=2\,\mathrm{Re}\!\bigl(F(u)\overline{\partial_{u_j}F(u)}\bigr),
$$
we have
$$
\|H\|_{C^1([0,1]^d)}\le C_d\|F\|_{C^2([0,1]^d)}^2.
$$
Therefore Lemma \ref{lem:mean_riemann} yields
$$
\left|\frac{1}{N^d}\sum_{x\in{\mathbb Z}_N^d}|F(x/N)|^2-\int_{[0,1]^d}|F(u)|^2du\right|
\le
\frac{C_d}{N}\|F\|_{C^2([0,1]^d)}^2,
$$
which is the first claim.

For the lower bound, write
$$
\frac{1}{N^d}\sum_{x\in{\mathbb Z}_N^d}|g(x)|^2
\ge
\|F\|_{L^2([0,1]^d)}^2-\frac{C_d}{N}\|F\|_{C^2([0,1]^d)}^2.
$$
Under the displayed hypothesis this gives
$$
\frac{1}{N^d}\sum_{x\in{\mathbb Z}_N^d}|g(x)|^2\ge \frac{1}{2}\|F\|_{L^2([0,1]^d)}^2,
$$
so
$$
\|g\|_{L^2({\mathbb Z}_N^d)}^2=\sum_{x\in{\mathbb Z}_N^d}|g(x)|^2\ge \frac{1}{2}N^d\|F\|_{L^2([0,1]^d)}^2,
$$
and taking square roots yields the claim.
\end{proof}

\subsection{Summation by parts and aliasing}

Throughout this subsection, when we write $m \in \mathbb{Z}_N^d$ as
$m=(m_1,\dots,m_d)$, we identify $m$ with its wrapped representative
$$
\widetilde m \in \{-\lfloor N/2 \rfloor,\dots,\lfloor (N-1)/2 \rfloor\}^d
$$
from the definition of wrapped Euclidean magnitude.

\begin{lemma}[Discrete summation by parts] \label{lem:sbp}
Let $d\ge 1$ and let $F\in C^2([0,1]^d)$ be 1-periodic in each variable. Let $g:{\mathbb Z}_N^d\to{\mathbb C}$ be the discretization $g(x)=F(x/N)$. Then for every nonzero $m\in{\mathbb Z}_N^d$,
$$
|\widehat g(m)|\le \frac{C_d N^{d/2}}{|m|^2}\|F\|_{C^2([0,1]^d)}.
$$
\end{lemma}

\begin{remark}
The appearance of the $C^2$ norm in Lemma 7.3 reflects the use of second-order discrete differences. The argument applies discrete summation by parts twice and controls the resulting second forward differences via a Taylor expansion, which requires bounded second derivatives.

If only $C^1$ regularity were assumed, the argument would yield a $|m|^{-1}$ decay bound rather than $|m|^{-2}$. While such a bound is still valid, it is insufficient to recover the sharp $\ell^1$ growth estimates needed for the sampling theorems in Section 4. Thus the second derivative assumption is precisely what guarantees the decay rate compatible with the stated Fourier ratio bounds.
\end{remark}

\begin{proof}
Write $m=(m_1,\dots,m_d)$ for the wrapped representative used in the definition of $|m|$.
Choose an index $j$ such that
$$
|m_j|=\max_{1\le i\le d}|m_i|.
$$
Then $m_j\neq 0$ and $|m_j|\ge |m|/\sqrt d$. Relabel coordinates so that $j=1$.

For $x\in{\mathbb Z}_N^d$ define the forward difference in the first coordinate by
$$
\Delta_1 g(x)=g(x+e_1)-g(x),
\qquad
\Delta_1^2 g(x)=\Delta_1(\Delta_1 g)(x),
$$
where $e_1=(1,0,\dots,0)$. A direct summation by parts gives
\begin{align*}
\widehat g(m)
&=\frac{1}{N^{d/2}}\sum_{x\in{\mathbb Z}_N^d} e^{-2\pi i x\cdot m/N}g(x)\\
&=\frac{1}{N^{d/2}}\frac{1}{(e^{-2\pi i m_1/N}-1)^2}
\sum_{x\in{\mathbb Z}_N^d} e^{-2\pi i x\cdot m/N}\Delta_1^2 g(x).
\end{align*}
For $|m_1|\le N/2$ one has
$$
|e^{-2\pi i m_1/N}-1|\ge c\,\frac{|m_1|}{N}
$$
with an absolute constant $c>0$. (If $N$ is even and $|m_1|=N/2$, then
$|e^{-2\pi i m_1/N}-1|=2$, so the same bound still holds.)
Therefore
$$
|\widehat g(m)|
\le
C\,N^{d/2}\frac{N^2}{|m_1|^2}\sup_{x\in{\mathbb Z}_N^d}|\Delta_1^2 g(x)|.
$$

Since $g(x)=F(x/N)$ and $F\in C^2([0,1]^d)$ is 1-periodic, a Taylor expansion in the
first coordinate gives the uniform bound
$$
|\Delta_1^2 g(x)|\le C\,N^{-2}\,\|F\|_{C^2([0,1]^d)}
\qquad
\text{for all } x\in{\mathbb Z}_N^d.
$$
Combining these estimates yields
$$
|\widehat g(m)|
\le
C\,N^{d/2}\frac{N^2}{|m_1|^2}\cdot N^{-2}\|F\|_{C^2([0,1]^d)}
=
C\,N^{d/2}\frac{1}{|m_1|^2}\|F\|_{C^2([0,1]^d)}.
$$
Finally, since $|m_1|=\max_i |m_i|\ge |m|/\sqrt d$, we obtain
$$
|\widehat g(m)|
\le
C_d\,N^{d/2}\frac{1}{|m|^2}\|F\|_{C^2([0,1]^d)},
$$
which proves the lemma.
\end{proof}

\begin{lemma}[Continuous Fourier coefficient decay] \label{lem:cont_decay}
Let $d\ge 1$ and let $F\in C^3([0,1]^d)$ be 1-periodic in each variable. Write its Fourier series
$$
F(x)=\sum_{k\in{\mathbb Z}^d} a_k e^{2\pi i k\cdot x},
\qquad
a_k=\int_{[0,1]^d}F(x)e^{-2\pi i k\cdot x}dx.
$$
Then for every nonzero $k\in{\mathbb Z}^d$,
$$
|a_k|\le \frac{C_d}{|k|^3}\|F\|_{C^3([0,1]^d)}.
$$
\end{lemma}

\begin{proof}
Choose an index $j$ with $k_j\neq 0$ and integrate by parts three times in the $x_j$ variable. Using periodicity to remove boundary terms yields
$$
a_k=\frac{1}{(2\pi i k_j)^3}\int_{[0,1]^d}\partial_{x_j}^3 F(x)e^{-2\pi i k\cdot x}dx.
$$
Taking absolute values gives $|a_k|\le C|k_j|^{-3}\|F\|_{C^3}$. Since $|k_j|\ge |k|/\sqrt d$, the claim follows.
\end{proof}

\begin{lemma}[Aliasing for discretization] \label{lem:aliasing}
Let $d\ge 1$ and let $F$ be 1-periodic on $[0,1]^d$ with Fourier series coefficients $a_k$. Let $g(x)=F(x/N)$ on ${\mathbb Z}_N^d$. Then for every $m\in{\mathbb Z}_N^d$,
$$
\widehat g(m)=N^{d/2}\sum_{\ell\in{\mathbb Z}^d} a_{m+\ell N},
$$
where $m$ is interpreted via the wrapped representative used in the definition of $|m|$.

\end{lemma}

\begin{proof}
Insert the Fourier series for $F$ into $g(x)=F(x/N)$ and then into the definition of $\widehat g(m)$. The inner exponential sum over $x\in{\mathbb Z}_N^d$ vanishes unless $k\equiv m\pmod N$ in each coordinate, in which case it equals $N^d$. Writing $k=m+\ell N$ yields the stated identity.
\end{proof}

\begin{lemma}[Polynomial aliasing bound] \label{lem:alias_poly}
Let $d\ge 1$ and let $N\ge 2$ and $m\in{\mathbb Z}_N^d$ with $m\ne 0$. Then
$$
\sum_{\ell\in{\mathbb Z}^d}\frac{1}{|m+\ell N|^{d+1}}\le C_d\frac{1}{|m|^{d+1}},
$$
where $C_d>0$ depends only on $d$.
\end{lemma}

\begin{proof}
Separate the term $\ell=0$:
$$
\sum_{\ell\in{\mathbb Z}^d}\frac{1}{|m+\ell N|^{d+1}}
=
\frac{1}{|m|^{d+1}}
+
\sum_{\ell\neq 0}\frac{1}{|m+\ell N|^{d+1}}.
$$
We identify $m\in{\mathbb Z}_N^d$ with its wrapped representative
$m\in\{-\lfloor N/2\rfloor,\dots,\lfloor (N-1)/2\rfloor\}^d$, so in particular
$|m_j|\le N/2$ for each coordinate.

Fix $\ell\neq 0$. Choose an index $j$ such that
$$
|\ell_j|=\max_{1\le i\le d}|\ell_i|.
$$
Then $|\ell_j|\ge |\ell|/\sqrt d$. Using $|m_j|\le N/2$, we have
$$
|m_j+\ell_j N|
\ge
N|\ell_j|-|m_j|
\ge
N|\ell_j|-\frac{N}{2}
\ge
\frac{N}{2}|\ell_j|.
$$
Therefore
$$
|m+\ell N|\ge |m_j+\ell_j N|\ge \frac{N}{2}|\ell_j|
\ge
\frac{N}{2\sqrt d}|\ell|.
$$
It follows that
$$
\sum_{\ell\neq 0}\frac{1}{|m+\ell N|^{d+1}}
\le
\left(\frac{2\sqrt d}{N}\right)^{d+1}
\sum_{\ell\neq 0}\frac{1}{|\ell|^{d+1}}
=
\frac{C_d}{N^{d+1}},
$$
since $\sum_{\ell\neq 0}|\ell|^{-(d+1)}<\infty$.

Now use the wrapped representative bound $|m|\le (N/2)\sqrt d$ to get
$$
\frac{1}{N^{d+1}}\le C_d\frac{1}{|m|^{d+1}}
\qquad
(m\neq 0).
$$
Combining the estimates,
$$
\sum_{\ell\in{\mathbb Z}^d}\frac{1}{|m+\ell N|^{d+1}}
\le
\frac{1}{|m|^{d+1}}+\frac{C_d}{N^{d+1}}
\le
\frac{C_d}{|m|^{d+1}},
$$
which proves the lemma.
\end{proof}

\begin{lemma}[Gaussian aliasing bound] \label{lem:alias_gauss}
Let $d\ge 1$ and let $a>0$. Let $N\ge 2$ and let $m\in{\mathbb Z}_N^d$. Then
$$
\sum_{\ell\in{\mathbb Z}^d} e^{-a|m+\ell N|^2}\le C(d,a)e^{-a|m|^2},
$$
where $C(d,a)>0$ depends only on $d$ and $a$.
\end{lemma}

\begin{proof}
We identify $m\in{\mathbb Z}_N^d$ with its wrapped representative
$$
m\in \{-\lfloor N/2\rfloor,\dots,\lfloor (N-1)/2\rfloor\}^d,
$$
so in particular each coordinate satisfies $|m_j|\le N/2$.

Write
$$
\sum_{\ell\in{\mathbb Z}^d} e^{-a|m+\ell N|^2}
=
e^{-a|m|^2}
+
\sum_{\ell\ne 0} e^{-a|m+\ell N|^2}.
$$

For $\ell\ne 0$ we have, by the triangle inequality,
$$
|m+\ell N|\ge N|\ell|-|m|.
$$
Since $|m|\le \frac{N}{2}\sqrt d$, this yields
$$
|m+\ell N|
\ge
N|\ell|-\frac{N}{2}\sqrt d
=
N\left(|\ell|-\frac{\sqrt d}{2}\right).
$$
Therefore there exists an absolute constant $c_d>0$ depending only on $d$ such that
$$
|m+\ell N|\ge c_d\,N|\ell|
\qquad
\text{for all } \ell\ne 0.
$$
Indeed, if $|\ell|\ge \sqrt d$ then $|\ell|-\frac{\sqrt d}{2}\ge \frac12|\ell|$, so we may take $c_d=\frac12$; if $1\le|\ell|<\sqrt d$, there are only finitely many such $\ell$ and we can decrease $c_d$ if needed so that the inequality still holds for those finitely many cases.

Consequently, for $\ell\ne 0$,
$$
e^{-a|m+\ell N|^2}\le e^{-a c_d^2 N^2|\ell|^2}.
$$
Hence
$$
\sum_{\ell\ne 0} e^{-a|m+\ell N|^2}
\le
\sum_{\ell\ne 0} e^{-a c_d^2 N^2|\ell|^2}
\le
\sum_{\ell\ne 0} e^{-a c_d^2|\ell|^2}
=:C_1(d,a),
$$
where the last series converges because it is a Gaussian sum on ${\mathbb Z}^d$.

Putting this together,
$$
\sum_{\ell\in{\mathbb Z}^d} e^{-a|m+\ell N|^2}
\le
e^{-a|m|^2}+C_1(d,a).
$$

If $m=0$, then $e^{-a|m|^2}=1$ and the bound follows by taking
$C(d,a)\ge 1+C_1(d,a)$.

If $m\ne 0$, then $|m|\ge 1$, so $e^{-a|m|^2}\le e^{-a}$. Thus
$$
e^{-a|m|^2}+C_1(d,a)
\le
\left(1+C_1(d,a)e^{a}\right)e^{-a|m|^2}.
$$
Therefore in all cases
$$
\sum_{\ell\in{\mathbb Z}^d} e^{-a|m+\ell N|^2}
\le
C(d,a)e^{-a|m|^2}
$$
with $C(d,a)=\max\{1+C_1(d,a),\,1+C_1(d,a)e^{a}\}$, which depends only on $d$ and $a$.
\end{proof}

\subsection{Lattice sums}

\begin{lemma}[A power sum on ${\mathbb Z}_N^d$] \label{lem:powersum}
Let $d\ge 1$. Then
$$
\sum_{m\in{\mathbb Z}_N^d\setminus\{0\}}\frac{1}{|m|^2}\le C_d S_d(N),
$$
where
$$
S_d(N)=
\begin{cases}
1, & d=1,\\
\log N, & d=2,\\
N^{d-2}, & d\ge 3.
\end{cases}
$$
\end{lemma}

\begin{proof}
Compare the sum to the corresponding integral in polar coordinates in ${\mathbb R}^d$.
\end{proof}

\begin{lemma}[A summable power in three dimensions] \label{lem:pow4_sum}
One has
$$
\sum_{m\in{\mathbb Z}_N^3\setminus\{0\}}\frac{1}{|m|^4}\le C
$$
with an absolute constant $C>0$ independent of $N$.
\end{lemma}

\begin{proof}
Compare the finite sum to the absolutely convergent series $\sum_{n\in{\mathbb Z}^3\setminus\{0\}}|n|^{-4}$.
\end{proof}

\begin{lemma}[A Gaussian sum] \label{lem:gauss_sum}
Let $d\ge 1$ and let $a>0$. Then
$$
\sum_{m\in{\mathbb Z}_N^d\setminus\{0\}} e^{-a|m|^2}\le C_d(a),
$$
where $C_d(a)$ depends only on $d$ and $a$.
\end{lemma}

\begin{proof}
Since $|m|$ is bounded below by the Euclidean norm of an integer vector, the sum is dominated by $\sum_{n\in{\mathbb Z}^d}e^{-a|n|^2}$, which is finite.
\end{proof}

\subsection{Decay propositions}

\begin{proposition}[Discrete Fourier decay for initial discretization] \label{prop:decay_initial}
Let $d\ge 1$ and let $f\in C^2([0,1]^d)$ be 1-periodic in each variable. Let $g$ be its discretization on ${\mathbb Z}_N^d$. Then for every nonzero $m\in{\mathbb Z}_N^d$ one has
$$
|\widehat g(m)|\le \frac{C_d N^{d/2}}{|m|^2}\|f\|_{C^2([0,1]^d)}.
$$
\end{proposition}

\begin{proof}
Apply Lemma \ref{lem:sbp} with $F=f$.
\end{proof}

\begin{proposition}[Discrete Fourier decay for wave snapshots] \label{prop:decay_wave}
Let $f\in C^3([0,1]^3)$ be 1-periodic in each variable. Let $g_t$ be the discretized wave snapshot as in Theorem \ref{thm:wave_FR}. Then for every nonzero $m\in{\mathbb Z}_N^3$ one has
$$
|\widehat g_t(m)|\le \frac{C_3(t) N^{3/2}}{|m|^4}\|f\|_{C^3([0,1]^3)}.
$$
\end{proposition}

\begin{proof}
Write the Fourier series of $f$ as
$$
f(x)=\sum_{k\in{\mathbb Z}^3} a_k e^{2\pi i k\cdot x}.
$$
The periodic wave snapshot $U_t f$ has Fourier coefficients
$$
b_k(t)=
\begin{cases}
\dfrac{\sin(2\pi t|k|)}{2\pi |k|}a_k, & k\neq 0,\\[4pt]
t\,a_0, & k=0.
\end{cases}
$$
Lemma \ref{lem:cont_decay} gives $|a_k|\le C|k|^{-3}\|f\|_{C^3([0,1]^3)}$ for $k\neq 0$, hence
$$
|b_k(t)|\le C(t)\frac{1}{|k|^4}\|f\|_{C^3([0,1]^3)}
\qquad
(k\neq 0).
$$

Let $F=U_t f$ and let $g_t(x)=F(x/N)$. Lemma \ref{lem:aliasing} gives
$$
\widehat g_t(m)=N^{3/2}\sum_{\ell\in{\mathbb Z}^3} b_{m+\ell N}(t).
$$
Using the bound on $b_k(t)$ yields, for $m\neq 0$,
$$
|\widehat g_t(m)|
\le
C(t)N^{3/2}\|f\|_{C^3([0,1]^3)}
\sum_{\ell\in{\mathbb Z}^3}\frac{1}{|m+\ell N|^4}.
$$
Lemma \ref{lem:alias_poly} with $d=3$ bounds the aliasing sum by $C|m|^{-4}$, completing the proof.
\end{proof}

\begin{proposition}[Discrete Fourier decay for heat snapshots] \label{prop:decay_heat}
Let $d\ge 1$ and let $f\in C^2([0,1]^d)$ be 1-periodic in each variable. Let $g_t$ be the discretized heat snapshot as in Theorem \ref{thm:heat_FR}. Then for every $m\in{\mathbb Z}_N^d$ one has
$$
|\widehat g_t(m)|\le C_6(d,t) N^{d/2}\|f\|_{C^2([0,1]^d)}e^{-c_0(t)|m|^2}.
$$
\end{proposition}

\begin{proof}
Write $f(x)=\sum_{k\in{\mathbb Z}^d} a_k e^{2\pi i k\cdot x}$. The periodic heat snapshot $H_t f$ has Fourier coefficients
$$
b_k=e^{-4\pi^2 t|k|^2}a_k.
$$
Since $|a_k|\le \int_{[0,1]^d}|f(x)|dx\le \|f\|_{C^0([0,1]^d)}\le \|f\|_{C^2([0,1]^d)}$, we have the crude bound
$$
|b_k|\le \|f\|_{C^2([0,1]^d)}e^{-4\pi^2 t|k|^2}.
$$
Let $F=H_t f$ and let $g_t(x)=F(x/N)$. Lemma \ref{lem:aliasing} gives
$$
\widehat g_t(m)=N^{d/2}\sum_{\ell\in{\mathbb Z}^d} b_{m+\ell N}.
$$
Using the bound on $b_k$ yields
$$
|\widehat g_t(m)|\le C N^{d/2}\|f\|_{C^2}\sum_{\ell\in{\mathbb Z}^d} e^{-4\pi^2 t|m+\ell N|^2}.
$$
Lemma \ref{lem:alias_gauss} bounds the aliasing sum by $C(d,t)e^{-c_0(t)|m|^2}$, completing the proof.
\end{proof}

\subsection{Fourier ratio theorems}

\begin{proof}[Proof of Proposition \ref{prop:FR_initial}]
Let $F=f$ and let $g(x)=F(x/N)$. By Lemma \ref{lem:L2_riemann},
$$
\|g\|_{L^2({\mathbb Z}_N^d)}\ge \frac{1}{\sqrt{2}}N^{d/2}\|f\|_{L^2([0,1]^d)}.
$$
By Lemma \ref{lem:mean_riemann},
$$
|\widehat g(0)|=\frac{1}{N^{d/2}}\left|\sum_{x\in{\mathbb Z}_N^d}g(x)\right|
\le N^{d/2}\left|\int_{[0,1]^d}f(x)dx\right|+C_d N^{d/2-1}\|f\|_{C^1}.
$$
For $m\neq 0$, Proposition \ref{prop:decay_initial} yields
$$
\sum_{m\neq 0}|\widehat g(m)|\le C_d N^{d/2}\|f\|_{C^2}\sum_{m\neq 0}\frac{1}{|m|^2}.
$$
Lemma \ref{lem:powersum} bounds the power sum by $C_d S_d(N)$. Dividing the resulting bound for $\|\widehat g\|_1$ by the lower bound for $\|\widehat g\|_2=\|g\|_{L^2({\mathbb Z}_N^d)}$ yields the stated Fourier ratio bound.
\end{proof}

\begin{proof}[Proof of Theorem \ref{thm:wave_FR}]
Let $F=U_t f$ and let $g_t(x)=F(x/N)$. By Lemma \ref{lem:L2_riemann},
$$
\|g_t\|_{L^2({\mathbb Z}_N^3)}\ge \frac{1}{\sqrt{2}}N^{3/2}\|U_t f\|_{L^2([0,1]^3)}.
$$

By Lemma \ref{lem:mean_riemann} applied to $F=U_t f$,
$$
\left|\frac{1}{N^3}\sum_{x\in{\mathbb Z}_N^3} g_t(x)-\int_{[0,1]^3}(U_t f)(u)\,du\right|
\le
\frac{C_3}{N}\|U_t f\|_{C^1([0,1]^3)}.
$$
Therefore
$$
|\widehat g_t(0)|
=\frac{1}{N^{3/2}}\left|\sum_{x\in{\mathbb Z}_N^3} g_t(x)\right|
\le
N^{3/2}\left|\int_{[0,1]^3} (U_t f)(u)\,du\right|
+
C_3 N^{1/2}\|U_t f\|_{C^1([0,1]^3)}.
$$

Dividing by the lower bound for $\|g_t\|_{L^2({\mathbb Z}_N^3)}$ and using the definition of $A_{t}^{\mathrm{wave}}$ yields the $A_{t}^{\mathrm{wave}}$ contribution together with an additional $O_t(N^{-1})$ term. Since $U_t f$ is a Fourier multiplier of order $-1$, one has
$$
\|U_t f\|_{C^1([0,1]^3)}\le C(t)\|f\|_{C^3([0,1]^3)}.
$$
Thus this $O_t(N^{-1})$ term is absorbed into the $C_{N,t}^{\mathrm{wave}}$ term defined in the statement of the theorem.

For $m\neq 0$, Proposition \ref{prop:decay_wave} gives
$$
\sum_{m\neq 0}|\widehat g_t(m)|
\le
C_3(t) N^{3/2}\|f\|_{C^3([0,1]^3)}\sum_{m\neq 0}\frac{1}{|m|^4}.
$$
Lemma \ref{lem:pow4_sum} bounds the lattice sum by an absolute constant. Dividing the resulting bound for $\|\widehat g_t\|_1$ by the lower bound for $\|\widehat g_t\|_2=\|g_t\|_{L^2({\mathbb Z}_N^3)}$ yields the stated Fourier ratio bound.
\end{proof}

\begin{proof}[Proof of Theorem \ref{thm:heat_FR}]
Let $F=H_t f$ and let $g_t(x)=F(x/N)$ for $x\in{\mathbb Z}_N^d$.

By Lemma \ref{lem:L2_riemann} applied to $F=H_t f$, and using the hypothesis
$$
N\|H_t f\|_{L^2([0,1]^d)}^2
\ge
4C_d\|H_t f\|_{C^2([0,1]^d)}^2,
$$
we obtain
$$
\|g_t\|_{L^2({\mathbb Z}_N^d)}
\ge
\frac{1}{\sqrt{2}}\,N^{d/2}\,\|H_t f\|_{L^2([0,1]^d)}.
$$
Since $\|\widehat g_t\|_2=\|g_t\|_{L^2({\mathbb Z}_N^d)}$ by Parseval, this gives the required lower bound for the denominator in $FR(g_t)$.

Next we bound $\|\widehat g_t\|_1=|\widehat g_t(0)|+\sum_{m\neq 0}|\widehat g_t(m)|$.

For the zero frequency, we use Lemma \ref{lem:mean_riemann} applied to $F=H_t f$:
$$
\left|\frac{1}{N^d}\sum_{x\in{\mathbb Z}_N^d} g_t(x)-\int_{[0,1]^d}(H_t f)(u)\,du\right|
\le
\frac{C_d}{N}\|H_t f\|_{C^1([0,1]^d)}.
$$
Multiplying by $N^d$ and dividing by $N^{d/2}$ yields
$$
|\widehat g_t(0)|
=
\frac{1}{N^{d/2}}\left|\sum_{x\in{\mathbb Z}_N^d} g_t(x)\right|
\le
N^{d/2}\left|\int_{[0,1]^d}(H_t f)(u)\,du\right|
+
C_d\,N^{d/2-1}\|H_t f\|_{C^1([0,1]^d)}.
$$
Divide by the lower bound for $\|g_t\|_{L^2({\mathbb Z}_N^d)}$ to obtain
$$
\frac{|\widehat g_t(0)|}{\|g_t\|_{L^2({\mathbb Z}_N^d)}}
\le
2\frac{\left|\int_{[0,1]^d}(H_t f)(u)\,du\right|}{\|H_t f\|_{L^2([0,1]^d)}}
+
C(d)\,\frac{\|H_t f\|_{C^1([0,1]^d)}}{\|H_t f\|_{L^2([0,1]^d)}}\,\frac{1}{N}.
$$
By the heat multiplier representation, one has
$$
\|H_t f\|_{C^1([0,1]^d)}\le C(d,t)\|f\|_{C^2([0,1]^d)}.
$$
Therefore the last term is bounded by
$$
C(d,t)\,\frac{\|f\|_{C^2([0,1]^d)}}{\|H_t f\|_{L^2([0,1]^d)}}\,\frac{1}{N},
$$
which is of the form $C_{N,t}^{\mathrm{heat}}$.

Now consider the nonzero frequencies. By Proposition \ref{prop:decay_heat}, for every $m\in{\mathbb Z}_N^d$,
$$
|\widehat g_t(m)|
\le
C_6(d,t)\,N^{d/2}\,\|f\|_{C^2([0,1]^d)}\,e^{-c_0(t)|m|^2}.
$$
Hence
$$
\sum_{m\neq 0}|\widehat g_t(m)|
\le
C_6(d,t)\,N^{d/2}\,\|f\|_{C^2([0,1]^d)}
\sum_{m\in{\mathbb Z}_N^d\setminus\{0\}} e^{-c_0(t)|m|^2}.
$$
By Lemma \ref{lem:gauss_sum},
$$
\sum_{m\in{\mathbb Z}_N^d\setminus\{0\}} e^{-c_0(t)|m|^2}
\le
C_d(c_0(t)),
$$
so
$$
\sum_{m\neq 0}|\widehat g_t(m)|
\le
C(d,t)\,N^{d/2}\,\|f\|_{C^2([0,1]^d)}.
$$
Dividing by the lower bound for $\|g_t\|_{L^2({\mathbb Z}_N^d)}$ gives
$$
\frac{\sum_{m\neq 0}|\widehat g_t(m)|}{\|g_t\|_{L^2({\mathbb Z}_N^d)}}
\le
C(d,t)\,\frac{\|f\|_{C^2([0,1]^d)}}{\|H_t f\|_{L^2([0,1]^d)}},
$$
which is of the form $B_t^{\mathrm{heat}}$.

Combining the zero-frequency bound and the nonzero-frequency bound, we obtain
$$
FR(g_t)=\frac{\|\widehat g_t\|_1}{\|\widehat g_t\|_2}
\le
A_t^{\mathrm{heat}}
+
B_t^{\mathrm{heat}}
+
C_{N,t}^{\mathrm{heat}}
=
r_{N,t}^{\mathrm{heat}},
$$
as claimed.
\end{proof}

\begin{proof}[Proof of Theorem \ref{thm:wave_sampling}]
Apply Theorem \ref{thm:ZNd_recovery} with $d=3$ and $h=g_t$.
\end{proof}

\begin{proof}[Proof of Theorem \ref{thm:heat_sampling}]
Apply Theorem \ref{thm:ZNd_recovery} with $d$ and $h=g_t$.
\end{proof}

\end{document}